\documentclass[11pt,reqno]{amsart}

\usepackage[margin=1.5in]{geometry}
\usepackage[colorinlistoftodos]{todonotes}

\newcounter{todocounter}

\usepackage[T1]{fontenc}
\usepackage{tgpagella}
\usepackage{mathpazo}

\usepackage[pagebackref]{hyperref} 
\usepackage{xcolor}
\usepackage{amsmath,amsthm}
\usepackage{amssymb}
\usepackage{graphicx}
\usepackage[mathscr]{eucal}
\usepackage{MnSymbol}
\usepackage[frame,ps,matrix,arrow,curve,rotate,all,2cell,tips]{xy}
\usepackage{epic,eepic}
\usepackage{enumerate}
\usepackage{enumitem}
\usepackage{color,soul}

\newtheorem{theorem}[subsection]{Theorem}
\newtheorem{lemma}[subsection]{Lemma}

\newtheorem{proposition}[subsection]{Proposition}
\newtheorem{corollary}[subsection]{Corollary}

\theoremstyle{definition}
\newtheorem{definition}[subsection]{Definition}
\newtheorem{example}[subsection]{Example}
\newtheorem{remark}[subsection]{Remark}
\newtheorem{assumption}[subsection]{Assumption}

\numberwithin{equation}{subsection}

\newcommand{\M}{\mathcal{M}}

\newcommand{\rkm}{R_K \M}

\newcommand{\C}{\mathcal{C}}

\newcommand{\D}{\mathcal{D}}

\renewcommand{\tilde}[1]{\widetilde{#1}}

\newcommand{\Ho}{\mathsf{Ho}}

\newcommand{\po}{\ar@{}[dr]|(.7){\Searrow}}
\newcommand{\pb}{\ar@{}[dr]|(.3){\Nwarrow}}

\DeclareMathOperator{\map}{map}
\newcommand{\mapm}{\map_{\M}}
\newcommand{\mapalgtm}{\map_{\algtm}}
\newcommand{\Ch}{\mathsf{Ch}}
\newcommand{\Chk}{\Ch(k)}

\usepackage{tikz}
\usetikzlibrary{arrows,decorations.pathmorphing}
\usetikzlibrary{backgrounds,positioning}
\usetikzlibrary{fit,petri,shapes.misc}

\tikzset{auto}

\tikzset{empty/.style={circle,inner sep=0pt,minimum size=6mm}}
\tikzset{emptyvt/.style={circle,inner sep=0pt,minimum size=0mm}}

\tikzset{plain/.style={circle,draw,very thick,
inner sep=0pt,minimum size=6mm}}

\tikzset{fatplain/.style={rounded rectangle,draw,very thick,minimum size=6mm}}

\tikzset{bigplain/.style={rounded rectangle,draw,very thick,minimum size=.8cm}}

\tikzset{yellowvt/.style={circle,draw,fill=yellow,very thick,inner sep=0pt,minimum size=6mm}}

\tikzset{bluevt/.style={circle,draw,fill=blue!20,very thick,inner sep=0pt,minimum size=6mm}}

\tikzset{greenvt/.style={circle,draw,fill=green!30,very thick,inner sep=0pt,minimum size=6mm}}

\tikzset{redvt/.style={circle,draw,fill=red!30,very thick,inner sep=0pt,minimum size=6mm}}

\tikzset{arrow/.style={->,thick}}
\tikzset{dashedarrow/.style={->,dashed,thick}}
\tikzset{dottedarrow/.style={->,dotted,thick}}
\tikzset{mapto/.style={|->,thick}}

\tikzset{implies/.style={thick,double,double equal sign distance,-implies}}

\tikzset{line/.style={thick}}
\tikzset{dottedline/.style={dotted,thick}}
\tikzset{dashedline/.style={dashed,thick}}

\tikzset{inputleg/.style={<-,thick}}
\tikzset{outputleg/.style={->,thick}}
\tikzset{dottedinput/.style={<-,dotted,thick}}

\newcommand{\adjoint}{
\nicearrow\xymatrix{ \ar@<2pt>[r] & \ar@<2pt>[l]}}
\renewcommand{\hookrightarrow}{\nicexy{\ar@{^{(}->}[r] &}}
\newcommand{\nicearrow}{\SelectTips{cm}{10}}
\newcommand{\nicexy}{\nicearrow\xymatrix@C+5pt}

\newcommand{\pushout}{\ar@{}[dr]|(0.75){\Searrow}}
\newcommand{\drrpushout}{\ar@{}[drr]|(0.90){\Searrow}}

\renewcommand{\to}{\hspace{-.1cm}\nicearrow\xymatrix@C-.2cm{\ar[r]&}\hspace{-.1cm}}

\newcommand{\sO}{\mathsf{O}}

\newcommand{\xtilde}{\widetilde{X}}

\newcommand{\uc}{\underline{c}}

\newcommand{\sset}{\mathsf{sSet}}
\newcommand{\Top}{\mathsf{Top}}

\newcommand{\alg}{\mathsf{Alg}}

\newcommand{\algtm}{\alg(T;\M)}
\newcommand{\algtrkm}{\alg(T;\rkm)}
\newcommand{\rtkalgtm}{R_{T(K)}\algtm}

\newcommand{\smallbinom}[2]
{\raisebox{.05cm}{\scalebox{0.8}{$\binom{#1}{#2}$}}}

\renewcommand{\lim}{\mathsf{lim}\,}

\DeclareMathOperator{\Id}{Id}

\newcommand{\deltan}{\Delta[n]}
\newcommand{\ddeltan}{\partial\Delta[n]}

\newcommand{\Lambdak}{\Lambda(K)}
\newcommand{\Lambdakbar}{\overline{\Lambdak}}

\renewcommand{\diamond}{\blacklozenge}

\newtheorem{main}{Theorem}

\begin{document}

\title{Right Bousfield Localization and Eilenberg-Moore Categories}

\author{David White}
\address{Department of Mathematics and Computer Science \\ Denison University
\\ Granville, OH}
\email{david.white@denison.edu}

\author{Donald Yau}
\address{Department of Mathematics \\ The Ohio State University at Newark \\ Newark, OH}
\email{dyau@math.osu.edu}

\begin{abstract}
We compare several recent approaches to studying right Bousfield localization and algebras over monads. We prove these approaches are equivalent, and we apply this equivalence to obtain several new results regarding right Bousfield localizations (some classical, some new) for spectra, spaces, equivariant spaces, chain complexes, simplicial abelian groups, and the stable module category. En route, we provide conditions so that right Bousfield localization lifts to categories of algebras, so that right Bousfield localization preserves algebras over monads, and so that right Bousfield localization forms a compactly generated model category.
\end{abstract}

\maketitle

\tableofcontents

\section{Introduction}

Bousfield localization is a valuable technique for creating new homotopy theories; it has been extensively studied in the context of model categories, infinity categories, triangulated categories, classical homotopy theory, and group theory. Both left and right Bousfield localization invert morphisms: in the context of model categories both results in a larger class of weak equivalences. Left Bousfield localization begins with a model category $\M$ and a set of maps $\C$, and results in a new model structure $L_\C \M$ on the same category, with the same cofibrations, where $\C$ is now contained in the new weak equivalences. Right Bousfield localization (also known as cellularization or colocalization) begins with a model category $\M$ and a set of objects $K$, and results in a new model structure $\rkm$, with the same fibrations, where $K$ is now contained in the class of cofibrant objects and where morphisms seen to be weak equivalences by $K$ (via homotopy function objects) are now weak equivalences. In both cases, hypotheses are needed on $\M$ to guarantee existence of the localization.

Left Bousfield localization is better behaved than right Bousfield localization. For example, $L_\C(\M)$ is guaranteed to be cofibrantly generated if $\M$ is, but the same is not true for $\rkm$. Nevertheless, right Bousfield localization plays an important role in homotopy theory dating back to CW approximation (more generally, $A$-cellular homotopy theory \cite{chacholski-thesis}), $n$-connected covers and Postnikov pieces \cite{nofech}, and finding point-set models in chain complexes and $R$-modules for localizing subcategories in the derived category of $R$ and the stable module category of $R$. A comprehensive list of applications can be found in \cite{white-yau2}. Due to the asymmetry between left and right Bousfield localization, more attention has been paid in the literature to left localization, and several different approaches to right localization have emerged (e.g. \cite{hirschhorn}, \cite{barwickSemi}, \cite{christensen-isaksen}).

This paper unifies several different approaches to questions related to preservation of algebraic structure under right Bousfield localization. All approaches considered are recalled here, so this paper can be read as a stand-alone paper. The approaches considered include transferring a semi-model structure from $\rkm$ to $\algtrkm$ in such a way that the forgetful functor preserves cofibrant objects \cite{white-yau2}, $R_K$ preserving $T$-algebras \cite{white-yau2}, lifting the right Bousfield localization functor $R_K$ to the homotopy category of $T$-algebras (where $T$ is a monad) \cite{crt}, and the forgetful functor taking $T(K)$-colocal equivalences to $K$-colocal equivalences \cite{grso}. The main goal is to prove the following theorem, which contains a converse to the main result of \cite{white-yau2}.

\begin{main} \label{thm:main}
Under Assumption \ref{standing.assumptions} suppose further that $\algtrkm$ is semi-admissible over $\rkm$.  Then the following statements are equivalent.
\begin{enumerate}
\item The forgetful functor
\[\nicexy{\algtrkm = \rtkalgtm \ar[r]^-{U} & \rkm}\]
preserves weak equivalences and cofibrant objects, in which the equality on the left is from Corollary \ref{cor:lifting.transfer}.
\item $R_K$ preserves $T$-algebras (Def. \ref{def:preserves-algebra}).
\item $R_K$ lifts to the homotopy category of $T$-algebras (Def. \ref{def:rk.lifts}).
\item The forgetful functor $U$ preserves right Bousfield localization (Def. \ref{def:grso.preserve}).
\end{enumerate}
\end{main}

In Section \ref{sec:liftingRight}, we provide preliminaries regarding right Bousfield localization and a proof of the equality in (1) above, where $T(K)$ denotes the free $T$-algebras on the objects $K$. Often, transferring model structures to categories of $T$-algebras such as $\algtrkm$ yields only semi-model category structures (see Example 2.8 in \cite{batanin-white} for a case where the transfer is provably not a model structure), so we recall semi-model categories in Section \ref{sec:liftingRight}, and we pay particular attention throughout this paper to highlighting the differences between semi-model categories and model categories throughout. Semi-admissibility in the theorem above refers to $\algtrkm$ having a semi-model structure, which is a strictly weaker condition than having a model structure.
In Section \ref{sec:admissibility}, we prove Theorem \ref{algtrkm.exists}, which provides conditions under which $\algtrkm$ has a transferred (semi-)model structure from $\rkm$, such that (1) holds. This result is of independent interest, as it allows for the study of algebras in a colocalized setting.
In Section \ref{sec:compactGeneration}, we prove a technical result regarding when $\rkm$ is compactly generated, a requirement for Theorem \ref{algtrkm.exists}. In Section \ref{sec:equivalentApproaches}, we recall the definitions required for (2), (3), and (4) above; then we prove Theorem \ref{thm:main}. Lastly, in Section \ref{sec:applications} we provide numerous applications of Theorem \ref{thm:main} to spectra, (equivariant) topological spaces, chain complexes, and the stable module category.

This paper can be viewed as the dual of \cite{batanin-white}, which unified approaches for left Bousfield localization, but the methods are far from formally dual. In particular, more care must be taken regarding the properties of $\rkm$ (e.g. Section \ref{sec:compactGeneration}). Additionally, there are instances where the asymmetry works to make right Bousfield localization \emph{easier} than left Bousfield localization. For example, Theorem \ref{lifting.transfer} (and Corollary \ref{cor:lifting.transfer}) proves that whenever one of the (semi-)model structures $\algtrkm$ or $\rtkalgtm$ exists, then both exist and coincide. The corresponding results are false for left Bousfield localization (see Remark 5.7 in \cite{batanin-white}). The key reason is an adjunction argument that gives control over the $T(K)$-colocal equivalences, whereas for left Bousfield localization we only had information about local objects.

\section{Lifting Right Bousfield Localization to Eilenberg-Moore Categories}
\label{sec:liftingRight}

In this paper, we will be transferring model structures to categories of $T$-algebras, for various monads $T$. In practice, there is often not a full model structure on $T$-algebras, but rather only a semi-model structure (see Example 2.8 in \cite{batanin-white}). The following definition is from \cite{spitzweck-thesis}, where it is called a $J$-semi model category over $\M$. This notion also appeared in \cite{barwickSemi}. A weaker notion of semi-model category has appeared in \cite{fresse-book}, analogous to what Spitzweck called an $(I,J)$-semi model category. The definition provided here is the most structure we can transfer to $T$-algebras, and all the results from \cite{fresse-book} remain true, since Fresse requires less structure on $\D$.

\begin{definition}
\label{defn:semi-model-cat}
Assume there is an adjunction $F:\M \adjoint \D:U$ where $\M$ is a cofibrantly generated model category, $\D$ is bicomplete, and the right adjoint $U$ preserves colimits over non-empty ordinals. 

We say that $\D$ is a \textbf{semi-model category} if $\D$ has three classes of morphisms called \emph{weak equivalences}, \emph{fibrations}, and \emph{cofibrations} such that the following axioms are satisfied. A \emph{cofibrant} object $X$ means an object in $\D$ such that the map from the initial object of $\D$ to $X$ is a cofibration in $\D$. Likewise, a \emph{fibrant} object is an object for which the map to the terminal object in $\D$ is a fibration in $\D$.

\begin{enumerate}
\item $U$ preserves and reflects fibrations and trivial fibrations ($=$ maps that are both weak equivalences and fibrations).
\item $\D$ satisfies the 2-out-of-3 axiom and the retract axiom.
\item Cofibrations in $\D$ have the left lifting property with respect to trivial fibrations. Trivial cofibrations ($=$ maps that are both weak equivalences and cofibrations) in $\D$ whose domain is cofibrant in $\M$ have the left lifting property with respect to fibrations.
\item Every map in $\D$ can be functorially factored into a cofibration followed by a trivial fibration. Every map in $\D$ whose domain is cofibrant in $\M$ can be functorially factored into a trivial cofibration followed by a fibration.
\item The initial object in $\D$ is cofibrant in $\M$.
\item Fibrations and trivial fibrations are closed under pullback.
\end{enumerate}

Denote by $I'$-inj the class of maps that have the right lifting property with respect to maps in $I'$. $\D$ is said to be \textit{cofibrantly generated} if there are sets of morphisms $I'$ and $J'$ in $\D$ such that the following conditions are satisfied.
\begin{enumerate}
\item
$I'$-inj is the class of trivial fibrations.
\item
$J'$-inj is the class of fibrations in $\D$.
\item
The domains of $I'$ are small relative to $I'$-cell.
\item
The domains of $J'$ are small relative to maps in $J'$-cell whose domain is sent by $U$ to a cofibrant object in $\M$.
\end{enumerate}
\end{definition}

Every model category is a semi-model category. A semi-model category in which all objects are cofibrant is a model category. In practice, anything that can be done in a model category can be done in a semi-model category, if one first cofibrantly replaces everything in sight. Several examples of the uses of semi-model categories can be found in \cite{goerss-hopkins-moduli-problems}, \cite{barwickSemi}, \cite{spitzweck-thesis}, \cite{ekmm}, \cite{white-localization}, \cite{white-commutative-monoids}, \cite{white-yau1}, \cite{white-yau2}, and \cite{gutierrez-rondigs-spitzweck-ostvaer} among other places.

Recall that a model category structure is completely determined by the classes of weak equivalences and of fibrations \cite{hirschhorn} (Prop. 7.2.7). The same is true for semi-model categories, since a map is a cofibration if and only if it satisfies the left lifting property with respect to all trivial fibrations \cite{barwickSemi} (Lemma 1.7). A left Quillen functor between semi-model categories is a functor whose right adjoint preserves fibrations and trivial fibrations.
 
\begin{definition}
Suppose $\M$ is a model category, $T$ is a monad on $\M$, and $\algtm$ is the category of $T$-algebras in $\M$.  We say that $\algtm$ is \emph{(semi-)admissible over $\M$} if it admits a (semi-)model category structure in which a map $f$ is a weak equivalence (resp., fibration) if $Uf \in \M$ is a weak equivalence (resp., fibration), where $U : \algtm \to \M$ is the forgetful functor.
\end{definition}

The homotopy function complex in a model category $\M$ is denoted $\mapm$.  Specifically, we will use the right homotopy function complex \cite{hirschhorn} (Def. 17.2.1).

\begin{definition}
Suppose $\M$ is a model category, and $K \subseteq \M$ is a set of cofibrant objects.  
\begin{enumerate}
\item A \emph{$K$-colocal equivalence} is a map $f : A \to B \in \M$ such that the induced map
\[\nicexy@C+1cm{\mapm(X,A) \ar[r]^-{\mapm(X,f)} & \mapm(X,B)}\]
is a weak equivalence of simplicial sets for all $X \in K$.
\item A \emph{$K$-colocal object} is a cofibrant object $Y$ in $\M$ such that the induced map
\[\nicexy@C+1cm{\mapm(Y,A) \ar[r]^-{\mapm(Y,f)} & \mapm(Y,B)}\]
is a weak equivalence of simplicial sets for all $K$-colocal equivalences $f : A \to B$.
\item Define a new category $\rkm$ as being the same as $\M$ as a category, together with the following distinguished classes of maps. A map $f$ in $\rkm$ is called a:
\begin{itemize}
\item \emph{weak equivalence} if it is a $K$-colocal equivalence.
\item \emph{fibration} if it is a fibration in $\M$.
\end{itemize}
\item If $\rkm$ is a model category with these weak equivalences and fibrations, then it is called the \emph{right Bousfield localization of $\M$ with respect to $K$} \cite{hirschhorn} (Def. 3.3.1(2) and Theorem 5.1.1).  In this case, $K$-colocal objects are precisely the cofibrant objects in $\rkm$.
\end{enumerate}
\end{definition}

\begin{assumption}\label{standing.assumptions}
Suppose $\M$ is a model category, $K \subseteq \M$ is a set of cofibrant objects such that the right Bousfield localization $\rkm$ exists, and $T$ is a monad on $\M$ such that $\algtm$ is semi-admissible over $\M$.
\end{assumption}

\begin{lemma}\label{tk.colocal.eq}
Under Assumption \ref{standing.assumptions} a map $f$ in $\algtm$ is a $T(K)$-colocal equivalence if and only if $Uf \in \M$ is a $K$-colocal equivalence.
\end{lemma}

\begin{proof}
Suppose $f : A \to B$ is a map of $T$-algebras in $\M$, and $f_\bullet : A_{\bullet} \to B_{\bullet}$ is a simplicial resolution of $f \in \algtm$ (\cite{hirschhorn} Def. 16.1.2(2) and Prop. 16.1.22(2) for model categories, \cite{barwickSemi} Theorem 3.12 for semi-model categories).  Applying the forgetful functor $U : \algtm \to \M$ entrywise to $f_{\bullet}$ yields a simplicial resolution of $Uf : UA \to UB \in \M$ because weak equivalences and fibrations in $\algtm$ are defined by the forgetful functor $U$.  

Pick any object $X \in K$.   Note that $T(X)$ is a cofibrant object in $\algtm$ because $X \in \M$ is cofibrant and $T : \M \to \algtm$ is a left Quillen functor.  There is a commutative diagram of simplicial sets using \cite{barwickSemi} (Scholium 3.64):
\[\nicexy@C+2cm{
\mapalgtm(T(X),A) \ar@{=}[d] \ar[r]^-{\mapalgtm(T(X),f)} & \mapalgtm(T(X),B) \ar@{=}[d]\\
\algtm\bigl(T(X),A_{\bullet}\bigr) \ar[d]_-{\cong} \ar[r]^-{\algtm(T(X),f_{\bullet})} & \algtm\bigl(T(X), B_{\bullet}\bigr) \ar[d]_-{\cong}\\
\M\bigl(X,(UA)_{\bullet}\bigr) \ar@{=}[d] \ar[r]^-{\M(X,Uf_{\bullet})} & \M\bigl(X,(UB)_{\bullet}\bigr) \ar@{=}[d]\\
\mapm(X,UA) \ar[r]^-{\mapm(X,Uf)} & \mapm(X,UB)}\] 
The middle isomorphisms are given by the free-forgetful adjunction between $\M$ and $\algtm$ and the remark above about $U(f_{\bullet}) = (Uf)_{\bullet}$ being a simplicial resolution of $Uf$.  Now the map $f \in \algtm$ is a $T(K)$-colocal equivalence if and only if the top horizontal map in the above diagram is a weak equivalence for all $X \in K$.  By commutativity this is equivalent to the bottom horizontal map being a weak equivalence for all $X \in K$.  This in turn is equivalent to $Uf$ being a $K$-colocal equivalence in $\M$.
\end{proof}

The following result has versions for both semi-model structures and full model structures. We prove the latter first, as the statement is valuable in its own right, and the proof will be easier for the reader to follow.

\begin{theorem}\label{lifting.transfer}
Under Assumption \ref{standing.assumptions} and the additional assumption that $T$ is admissible over $\M$, the following two statements are equivalent.
\begin{enumerate}
\item $\algtrkm$ is admissible over $\rkm$.
\item The right Bousfield localization $\rtkalgtm$ exists.
\end{enumerate}
Furthermore, if either statement is true, then the model categories $\algtrkm$ and $\rtkalgtm$ are equal.
\end{theorem}

Theorem \ref{lifting.transfer} says that, in the diagram
\begin{equation}\label{fundamental.diagram}
\nicexy@C+1cm{\algtm \ar@{|.>}[r]^-{R_{T(K)}}_-{\text{exists $?$}} & \rtkalgtm \ar@{=}[r]^-{?} & \algtrkm\\
\M \ar@{|->}[u]^-{\text{transfer}} \ar@{|->}[rr]^-{R_K} && \rkm \ar@{|.>}[u]^-{\text{transfer}}_-{\text{exists $?$}}}
\end{equation}
the ability to go counter-clockwise is equivalent to the ability to go clockwise.  Furthermore, when either one is possible, the results are equal.

\begin{proof}
The categories $\rtkalgtm$ and $\algtrkm$ are both equal to the category $\algtm$.  Observe that $\algtrkm$ and $\rtkalgtm$ have the same fibrations, namely, fibrations in $\algtm$.  Moreover, Lemma \ref{tk.colocal.eq} says that $\algtrkm$ and $\rtkalgtm$ have the same weak equivalences. 
\end{proof}

\begin{remark}
In \cite{crt} (7.7) the equality
\begin{equation}\label{same.model}
\algtrkm = \rtkalgtm
\end{equation}
of model categories was observed under assumption (2)--i.e., $\rtkalgtm$ exists--\emph{and} that $R_K$ lifts to the homotopy category of $T$-algebras (see Def. \ref{def:rk.lifts}).  Moreover, the same equality was also observed in \cite{grso} when $T$ is the monad of a colored operad.
\end{remark}

The proof of Theorem \ref{lifting.transfer} also proves the following corollary, since Lemma \ref{tk.colocal.eq} also holds for semi-model categories. By (2) below, we mean that the three classes of maps defining $\rtkalgtm$ \cite{hirschhorn} satisfy the axioms of a semi-model category.

\begin{corollary}\label{cor:lifting.transfer}
Under Assumption \ref{standing.assumptions}, the following two statements are equivalent.
\begin{enumerate}
\item $\algtrkm$ is semi-admissible over $\rkm$.
\item The right Bousfield localization $\rtkalgtm$ exists as a semi-model category.
\end{enumerate}
Furthermore, if either statement is true, then the semi-model categories $\algtrkm$ and $\rtkalgtm$ are equal.
\end{corollary}

Conditions under which (1) holds (hence (2) as well) are provided in Theorem \ref{algtrkm.exists} below. Conditions under which (2) has the structure of a \emph{right} semi-model category (dual to Definition \ref{defn:semi-model-cat}) are provided in \cite{barwickSemi}, but we do not make use of this structure. As conditions regarding semi-model category existence are more easily verified than full model category existence, we prefer to work in the setting of semi-model categories, and the conclusions of Corollary \ref{cor:lifting.transfer} are sufficient for our needs.

\begin{corollary}
Under Assumption \ref{standing.assumptions} suppose either one of the two equivalent conditions in Corollary \ref{cor:lifting.transfer} is true.  Then the functor $T : \M \to \algtm$ sends $K$-colocal equivalences between $K$-colocal objects in $\M$ to $T(K)$-colocal equivalences between $T(K)$-colocal objects in $\algtm$.
\end{corollary}

\begin{proof}
By Corollary \ref{cor:lifting.transfer} the free-forgetful adjunction
\[\nicexy{\rkm \ar@<2pt>[r]^-{T} & \algtrkm = \rtkalgtm \ar@<2pt>[l]^-{U}}\]
is a Quillen adjunction.  By Ken Brown's Lemma (\cite{hovey} 1.1.12 and \cite{fresse-book} 12.1.7) the left Quillen functor $T$ sends weak equivalences between cofibrant objects to weak equivalences between cofibrant objects.
\end{proof}

\section{Admissibility over Right Bousfield Localization}
\label{sec:admissibility}

\begin{definition}
Suppose $\M$ is a category, $C$ is a class of maps in $\M$, and $T$ is a monad on $\M$.
\begin{enumerate}
\item We call the class $C$ \emph{saturated} if it is closed under retracts, pushouts, and transfinite compositions.
\item We call $T$ \emph{finitary} if it preserves filtered colimits.
\end{enumerate}
\end{definition}

\begin{definition}
Suppose $\M$ is a model category, $C$ is a saturated class of maps in $\M$, and $T$ is a monad on $\M$.
\begin{enumerate}
\item We say that $\M$ is \emph{$C$-perfect} \cite{batanin-berger} (Def. 2.1) if weak equivalences in $\M$ are closed under filtered colimits along maps in $C$.
\item We say that $\M$ is \emph{$C$-compactly generated} \cite{batanin-berger} (Def. 2.4) if it is cofibrantly generated and $C$-perfect and if every object is small with respect to $C$ \cite{hovey} (Def. 2.1.3).
\item We say that $T$ is \emph{$C$-admissible on $\M$} \cite{batanin-berger} (Def. 2.9) if for each cofibration (resp., trivial cofibration) $f : A \to B \in \M$ and each pushout of the form
\begin{equation}\label{freemap.pushout}
\nicexy{T(A) \ar[d]_-{Tf} \ar[r] & X \ar[d]^-{g}\\ T(B) \ar[r] & Y}
\end{equation}
in $\algtm$, the underlying map $Ug \in \M$ is in $C$ (resp., $C$ and the class of weak equivalences), where $U : \algtm \to \M$ is the forgetful functor.
\item We say that $T$ is \emph{$C$-semi-admissible on $\M$} \cite{batanin-white} (Def. 2.4) if the previous statement holds whenever $UX$ is cofibrant in $\M$.
\end{enumerate}
\end{definition}

The following observation provides conditions under which the equivalent statements in Corollary \ref{cor:lifting.transfer} are true.

\begin{theorem}\label{algtrkm.exists}
Suppose:
\begin{itemize}
\item $\M$ is a cofibrantly generated model category.
\item $C \subseteq \M$ is a saturated class of maps such that $\M$ is $C$-compactly generated.
\item $K \subseteq \M$ is a set of cofibrant objects such that $\rkm$ exists and is $C$-compactly generated.
\item $T$ is a finitary $C$-(semi-)admissible monad on $\M$.  
\end{itemize}
Then $\algtrkm$ is (semi-)admissible over $\rkm$.
\end{theorem}

\begin{proof}
First consider the case of full model categories.  To show that $\algtrkm$ is admissible, by \cite{batanin-berger} (Theorem 2.11) it is enough to show that $T$ is $C$-admissible on $\rkm$.  Suppose $f : A \to B \in \rkm$, and consider the pushout \eqref{freemap.pushout}.
\begin{enumerate}
\item If $f$ is a trivial cofibration in $\rkm$, then it is also a trivial cofibration in $\M$ because $\M$ and $\rkm$ have the same fibrations and hence the same trivial cofibrations.  Since $T$ is $C$-admissible on $\M$, in the pushout \eqref{freemap.pushout} the map $Ug$ is in $C$ and is a weak equivalence in $\M$, hence also a weak equivalence in $\rkm$.
\item If $f$ is a cofibration in $\rkm$, then it is also a cofibration in $\M$.  So $T$ being $C$-admissible on $\M$ implies that the map $Ug$ is in $C$.
\end{enumerate} 
For the semi-model category case, we use the same argument with \cite{fresse} (12.1.4 and 12.1.9) instead of \cite{batanin-berger} (2.11) and assume that $UX$ is cofibrant in $\rkm$, hence also cofibrant in $\M$.
\end{proof}

\begin{corollary}
Under the assumptions of Theorem \ref{algtrkm.exists} in which $T$ is $C$-semi-admissible on $\M$, conditions (1) and (2) in Corollary \ref{cor:lifting.transfer} and the equality \eqref{same.model} are all true.
\end{corollary}

\begin{proof}
$\algtm$ is admissible over $\M$ by \cite{batanin-berger} (2.11), so Assumption \ref{standing.assumptions} is satisfied.  As $\algtrkm$ is semi-admissible by Theorem \ref{algtrkm.exists}, condition (1) in Corollary \ref{cor:lifting.transfer} is true, hence so are condition (2) and \eqref{same.model}.
\end{proof}

\section{Compact Generation of Right Bousfield Localization}
\label{sec:compactGeneration}

In Theorem \ref{algtrkm.exists}, the assumption that $\rkm$ is $C$-compactly generated means that:
\begin{enumerate}
\item \emph{Every object is small with respect to $C$.}  This is part of the assumptions of $\M$ being $C$-compactly generated.
\item \emph{$\rkm$ is a cofibrantly generated model category.}  For example, if $\M$ is cellular in which every object is fibrant, then $\rkm$ is also a cellular, hence cofibrantly generated, model category by \cite{hirschhorn} Theorem 5.1.1.
\item \emph{The class of $K$-colocal equivalences is closed under filtered colimits along maps in $C$.}  Below we will provide reasonable conditions under which this is true.
\end{enumerate}

\begin{definition}
Suppose $\M$ is a model category with a distinguished set of maps $J$, and $K$ is a set of objects in $\M$.
\begin{enumerate}
\item Define the set of maps
\[\Lambda(K) = \Bigl\{A^{\bullet} \otimes \ddeltan \to A^{\bullet} \otimes \deltan : A \in K,\, n \geq 0 \Bigr\}\]
in which $A^{\bullet}$ is a choice of a cosimplicial resolution of $A$ \cite{hirschhorn} (Def. 16.1.2(1) and 16.3.1(1)).
\item Define the set $\Lambdakbar = \Lambda(K) \cup J$ \cite{hirschhorn} (Def. 5.2.1).
\end{enumerate}
\end{definition}

$\rkm$ can be made cofibrantly generated even if $\M$ is not cellular. One approach is given in \cite{christensen-isaksen} (Theorem 2.6). We provide here another approach, resulting in a combinatorial model structure on $\rkm$.

\begin{proposition}
Suppose $\M$ is a combinatorial model category, $K$ is a set of objects, and every object of $\M$ is fibrant. Then $\rkm$ is a combinatorial model category in which every object is fibrant.
\end{proposition}

\begin{proof}
Existence of the model structure on $\rkm$ is Proposition 5.13 in \cite{barwickSemi}, and all objects are fibrant because the fibrations in $\rkm$ are the same as the fibrations in $\M$. A characterization of the cofibrations in $\rkm$, that matches the characterization in Proposition 5.3.6 of \cite{hirschhorn}, is given in \cite{barwickSemi} (Lemma 5.8). Lastly, \cite{christensen-isaksen} (Lemma 2.3) proves that $\rkm$ is cofibrantly generated, with trivial cofibrations as in $\M$ and generating cofibrations $\overline{\Lambda(K)}$, if all objects are fibrant. The key point here is that cellularity is not required for the arguments regarding $\overline{\Lambda(K)}$ \cite{hirschhorn} (5.2.4-5.2.6) to work.
\end{proof}

\begin{proposition}\label{rkm.cg}
Suppose:
\begin{itemize}
\item $\M$ is a cofibrantly generated simplicial model category with the set $J$ of generating trivial cofibrations, and $C$ is a saturated class of maps in $\M$.
\item $K$ is a set of cofibrant objects in $\M$ such that $\rkm$ is a cofibrantly generated model category with generating cofibrations $\Lambdakbar$. 
\item All the objects in $K \otimes \deltan$ for $n \geq 0$ and all the (co)domains of the maps in $J$ are finite with respect to $C$ \cite{hovey} (7.4).
\end{itemize}
Then $K$-colocal equivalences are closed under filtered colimits along $C$. It follows that $\rkm$ is $C$-compactly generated if all objects are small with respect to $C$ (e.g. if $\M$ is locally presentable).
\end{proposition}

\begin{proof}
By Hovey's argument \cite{hovey} (Cor. 7.4.2) (see also \cite{batanin-berger} Remark 2.2) it suffices to show that the domains and codomains of the maps in $\Lambdakbar = \Lambdak \cup J$ are finite with respect to $C$.  This is true for the maps in $J$ by assumption.  For $\Lambdak$ suppose $A \in K$.  Using the simplicial model structure of $\M$, since $A$ is cofibrant in $\M$, by \cite{hirschhorn} (Cor. 16.1.4(1)) a choice of a cosimplicial resolution of $A$ is given by
\[A^{\bullet} = \bigl\{A \otimes \deltan : n \geq 0 \bigr\}.\]
In order for $A^{\bullet} \otimes \ddeltan$ and $A^{\bullet} \otimes \deltan$--both of which are finite colimits of the various $A \otimes \deltan$ \cite{hirschhorn} (Def. 16.3.1(1))--to be $C$-finite, it is enough for $A \otimes \deltan$ to be $C$-finite, which is true by assumption.
\end{proof}

\section{Equivalent Approaches to Preservation of Algebras Under Right Bousfield Localization}
\label{sec:equivalentApproaches}

The next definition is \cite{crt} (Def. 7.9) for right Bousfield localization.  It provides one approach to preservation of monadic algebras under right Bousfield localization.

\begin{definition}\label{def:rk.lifts}
Under Assumption \ref{standing.assumptions} we say that \emph{$R_K$ lifts to the homotopy category of $T$-algebras} if:
\begin{enumerate}
\item There exists a coaugmented endofunctor $c^T : C^T \to \Id$ on $\Ho(\algtm)$.
\item There exists a natural isomorphism $h : R_KU \to UC^T$ such that
\[Uc^T \circ h = cU\]
in $\Ho(\M)$, where $c : R_K \to \Id$ is the counit of the derived adjunction $\Ho(\rkm) \adjoint \Ho(\M)$ and $U$ is the forgetful functor.
\end{enumerate}
\end{definition}

Another approach to preservation of monadic algebras under right Bousfield localization was proposed by the authors in \cite{white-yau2}, from which the following definition can be extracted.

\begin{definition}\label{def:preserves-algebra}
Under Assumption \ref{standing.assumptions}  we say that \textit{$R_K$ preserves $T$-algebras} if:
\begin{enumerate}
\item When $X$ is a $T$-algebra there is some $T$-algebra $\tilde{X}$ that is weakly equivalent in $\M$ to $R_K UX$.
\item In addition, when $X$ is a fibrant $T$-algebra, there is a choice of $\tilde{X}$ in $\algtm$, with $U(\tilde{X})$ colocal in $\M$, there is a $T$-algebra homomorphism $c_X:\tilde{X}\to X$ lifting the colocalization map $q:R_K UX\to UX$ up to homotopy, and there is a weak equivalence $\beta_X:U(\tilde{X}) \to R_K UX$ such that $q\circ \beta_X \cong U c_X$ in Ho$(\M)$.
\end{enumerate}
\end{definition}

A related concept is whether the forgetful functor preserves right Bousfield localization in the following sense.

\begin{definition}\label{def:grso.preserve}
Under Assumption \ref{standing.assumptions} suppose $\rtkalgtm$ exists as a semi-model category.  Then we say that the forgetful functor
\[\nicexy{\rtkalgtm \ar[r]^-{U} & \rkm}\]
\emph{preserves right Bousfield localization} if, given any map $c : R_{T(K)}X \to X \in \algtm$ that is a $T(K)$-colocal equivalence with $T(K)$-colocal domain, the map $Uc \in \M$ is a $K$-colocal equivalence with $K$-colocal domain. 
\end{definition}

We now observe that these three approaches to preservation of algebras under right Bousfield localization are equivalent.  The following omnibus theorem should be compared to \cite{batanin-white} (Theorem 5.6), which deals with different approaches to preservation of algebras under left Bousfield localization.

\begin{theorem}\label{equivalent.preservation}
Under Assumption \ref{standing.assumptions} suppose further that $\algtrkm$ is semi-admissible over $\rkm$.  Then the following statements are equivalent.
\begin{enumerate}
\item The forgetful functor
\[\nicexy{\algtrkm = \rtkalgtm \ar[r]^-{U} & \rkm}\]
preserves weak equivalences and cofibrant objects, in which the equality on the left is from Corollary \ref{cor:lifting.transfer}.
\item $R_K$ preserves $T$-algebras (Def. \ref{def:preserves-algebra}).
\item $R_K$ lifts to the homotopy category of $T$-algebras (Def. \ref{def:rk.lifts}).
\item The forgetful functor $U$ preserves right Bousfield localization (Def. \ref{def:grso.preserve}).
\end{enumerate}
\end{theorem}

\begin{proof}
(1) $\Longrightarrow$ (2) is proven in Theorem 6.2 in \cite{white-yau2}. To be self-contained, we recall the details in the case when $X$ is a fibrant $T$-algebra. We define $\tilde{X}$ to be $Q_{K,T} Q_T X$ where $Q_T$ (resp. $Q_{K,T})$ denotes the cofibrant replacement functor in $\algtm$ (resp. $\algtrkm$), and the map $c_X$ is simply the composite of cofibrant replacement maps $Q_{K,T} Q_T X \to Q_T X \to X$. The map $\beta_X$ is defined by the following lifting diagram in $\rkm$, where the right vertical map is cofibrant replacement in $\rkm$:
\[
\nicexy@R+.5cm@C+.5cm{\varnothing \ar@{>->}[d] \ar@{>->}[r] & R_KUX \ar@{->>}[d]^-{q}_-{\simeq}\\ UQ_{K,T} Q_T X \ar[r]^-{Uc_X}_-{\simeq} \ar@{.>}[ur]^-{\beta} & UX}
\]
Using that $U$ preserves cofibrant objects we conclude that $UQ_{K,T} Q_T X$ is cofibrant in $\rkm$. It easily follows that $\beta$ is a weak equivalence, since it is a $K$-colocal equivalence between $K$-colocal objects (by the 2-out-of-3 property).

For (2) $\Longrightarrow$ (3) we use the proof above that (1) $\Longrightarrow$ (2). We take $C^T$ to be the image in $\Ho(\algtm)$ of $\tilde{(-)}$ (i.e. of $Q_{K,T} Q_T (-)$), so that $c^T$ is the image in Ho$(\M)$ of $c_X$. To construct $h:R_K U \to U C^T$, consider the following lifting diagram, where $X$ is a fibrant $T$-algebra:
\[
\nicexy@R+.5cm@C+.5cm{\varnothing \ar@{>->}[d] \ar@{>->}[r] & UQ_{K,T} Q_T X \ar@{->>}[d]^-{Uc_X}_-{\simeq}\\ R_KUX \ar[r]^-{q}_-{\simeq} \ar@{.>}[ur]^-{\beta} & UX}
\]
The left vertical map above is a cofibration in $\rkm$, because $q$ is a cofibrant replacement map in $\rkm$. The right vertical map is a trivial fibration in $\rkm$ because $U$ preserves trivial fibrations and $c_X$ is a trivial fibration in $\algtrkm$. It follows that $\beta$ is a weak equivalence in $\M$. We take $h$ to be the image of $\beta$ in Ho$(\M)$, and immediately deduce that $h$ is an isomorphism in Ho$(\M)$ and that $Uc^T \circ h = cU$ in Ho$(\M)$ (see Def. \ref{def:rk.lifts}) by commutativity of the lower triangle above, since $c$ is the image of $q$ in Ho$(\M)$. Furthermore, the lift $\beta$ is unique in Ho$(\M$) by the universal property of right localization, since any other lift would necessarily be a weak equivalence between the same $K$-colocal objects. Finally, this lift is natural in Ho$(\M)$ because if we began with a map $X\to Y$ and constructed lifts $\beta_X$ and $\beta_Y$ then we would in addition construct a homotopy unique lift from $R_K UX$ to $U Q_{K,T} Q_T Y$, so commutativity of the relevant naturality square in Ho($\M)$ follows from uniqueness.

The implication (3) $\Longrightarrow$ (1) is \cite{crt} (Theorem 7.10(iii)) in the case of model categories.  For semi-model categories, to see that the forgetful functor preserves cofibrant objects, we could just use the assumed natural isomorphism $h : R_KU \to UC^T$ in Def. \ref{def:rk.lifts}(2) and the equality $\algtrkm = \rtkalgtm$ in Corollary \ref{cor:lifting.transfer}.  The forgetful functor preserves weak equivalences because $\algtrkm$ is semi-admissible over $\rkm$.  We have shown that the first three statements are equivalent.

To see that (1) $\Longrightarrow$ (4), simply observe that the map $c$ in Def. \ref{def:grso.preserve} is a weak equivalence with cofibrant domain in $\rtkalgtm$.  So (1) guarantees that the map $Uc \in \rkm$ is also a  weak equivalence with cofibrant domain, i.e., a $K$-colocal equivalence with $K$-colocal domain in $\M$.

Finally, we show that (4) $\Longrightarrow$ (2).  Since $\algtrkm = \rtkalgtm$ exists by Corollary \ref{cor:lifting.transfer}, every $T$-algebra $X$ has a functorial cofibrant replacement in $\rtkalgtm$.  In other words, we may take the given functorial factorization
\[\nicexy{\varnothing \ar@{>->}[r] & R_{T(K)}X \ar@{>>}[r]^-{c_X}_-{\simeq} & X \in \rtkalgtm}\]
of the map $\varnothing \to X$ into a cofibration followed by a trivial fibration in $\rtkalgtm$.  We take $R_{T(K)}X \in \algtm$ as our choice of $\xtilde$ in Def.  \ref{def:preserves-algebra}.  Since the above factorization is taken in $\rtkalgtm$, the map $c_X$ is a $T$-algebra map.  Furthermore, (4) guarantees that the underlying map $Uc_X \in \M$ is a $K$-colocal equivalence with $K$-colocal domain.

It remains to show that $\xtilde = R_{T(K)}X$ is weakly equivalent to $R_KUX$ in $\M$.  Consider the diagram
\begin{equation} \label{lift-beta}
\nicexy@R+.5cm@C+.5cm{\varnothing \ar@{>->}[d] \ar@{>->}[r] & R_KUX \ar@{->>}[d]^-{q}_-{\simeq}\\ UR_{T(K)}X \ar[r]^-{Uc_X}_-{\simeq} \ar@{.>}[ur]^-{\beta} & UX}
\end{equation}
in $\rkm$, in which $q : R_KUX \to UX$ is the given functorial cofibrant replacement of $UX$ in $\rkm$.  Since the map $q$ is a trivial fibration and $UR_{T(K)}X$ is cofibrant in $\rkm$, a dotted lift $\beta$ exists, making the lower triangle commute as required.  Since both $Uc_X$ and $q$ are weak equivalences in $\rkm$, so is $\beta$ by the $2$-out-of-$3$ property.  Therefore, $\beta$ is a $K$-colocal equivalence between $K$-colocal objects in $\M$.  So it is actually a weak equivalence in $\M$ by \cite{hirschhorn} (3.2.13(2)), as desired.
\end{proof}

\begin{remark}
In the context of a colored operad in a cofibrantly generated simplicial monoidal model category, the implication (1) $\Longrightarrow$ (4) in Theorem \ref{equivalent.preservation} is a result in \cite{grso}.
\end{remark}

\section{Applications}
\label{sec:applications}

We conclude this paper with numerous applications of Theorem \ref{equivalent.preservation}. First, in \cite{crt} it was proven that certain colocalizations lift to the homotopy category category of modules over a ring spectrum. We can now use Theorem \ref{equivalent.preservation} to deduce that these colocalizations preserve ring and module structure. As the setting of spectra eluded us in \cite{white-yau2}, these are the first general results we are aware of regarding preservation of algebraic structure by right Bousfield localizations for spectra.

Next, several results regarding preservation of algebraic structure by various right Bousfield localizations were proven in \cite{white-yau2} for chain complexes, spaces, equivariant spaces, and the stable module category. We can now use Theorem \ref{equivalent.preservation} to deduce that these right Bousfield localizations lift to homotopy categories of $T$-algebras and that the forgetful functor preserves these right Bousfield localizations.

\subsection{Spectra}

Let $\M$ be the positive stable model structure on symmetric spectra. Let $K$ be a cofibrant spectrum. 

\begin{example}
If $E$ be a connective ring spectrum that is cofibrant in $\M$, then $R_K$ preserves $E$-modules. This follows from the proof of Theorem 7.11 in \cite{crt} and Theorem \ref{equivalent.preservation}, since $R_K$ lifts to the homotopy category of $E$-modules.
\end{example}

\begin{example}
If $K$ is connective, and $E$ is any ring spectrum that is cofibrant in $\M$, then $R_K$ preserves $E$-modules. This follows from the proof of Theorem 7.11 in \cite{crt} and Theorem \ref{equivalent.preservation}, since $R_K$ lifts to the homotopy category of $E$-modules.
\end{example}

\begin{theorem} \label{thm:spectra-preservation}
Let $P$ be a colored operad valued in simplicial sets or symmetric spectra. Let $K$ be a set of cofibrant objects such that $\rkm$ is a monoidal model category and the $K$-colocal objects are closed under smash product (e.g. if $R_K$ is an $\M$-\textit{enriched} colocalization in the sense of Definition 5.2 of \cite{gutierrez-transfer-quillen}). Then $R_K$ preserves $P$-algebras, $R_K$ lifts to the homotopy category of $P$-algebras, and $U$ preserves $R_K$.
\end{theorem}

\begin{proof}
The key ingredient will be Theorem 5.8 in \cite{gutierrez-transfer-quillen}, together with the dual of Theorem 7.6 in \cite{crt}. We have slightly changed the notation from \cite{gutierrez-transfer-quillen} to match our notation. We focus first on the case of simplicial colored operads. Let $\phi:P_\infty \to P$ be a cofibrant replacement in the model category of simplicial operads. By Theorems 1.3 and 1.4 in \cite{elmendorf-mandell}, the categories of $P_\infty$-algebras and $P$-algebras both inherit transferred model structures, and they are Quillen equivalent via $\phi^*$. Thus, preservation for $P$-algebras is equivalent to preservation for $P_\infty$-algebras.

Let $X$ be a $P_\infty$-algebra and let $r:X\to X'$ be fibrant replacement in the model category of $P_\infty$-algebras (so $X'$ is colorwise fibrant as a spectrum). For every set of colors $c_1,\dots,c_n$, the object $R_K(U(X'(c_1))) \otimes \dots \otimes R_K(U(X'(c_n)))$ is $K$-colocal by our hypothesis on $K$-colocal objects. It follows from Theorem 5.8 in \cite{gutierrez-transfer-quillen} that $R_K(U(X'))$ admits a $P_\infty$-algebra structure. Furthermore, $R_K(U(X'))$ is $K$-colocally weakly equivalent to $R_K(U(X))$, by the 2-out-of-3 property, hence weakly equivalent by Theorem 3.2.13 of \cite{hirschhorn}. It follows that $R_K$ preserves $P$-algebras, hence (by Theorem \ref{equivalent.preservation}) that $R_K$ lifts to the homotopy category of $P$-algebras and that $U$ preserves $R_K$. 

For the case of spectral colored operads, one cannot use the results in \cite{elmendorf-mandell}. However, admissibility of $P$ and $P_\infty$ follow from Theorem 8.3.1 in \cite{white-yau1}, and rectification between $P$-algebras and $P_\infty$-algebras can be verified as in Corollary 5.1.1 of \cite{white-yau3} or Theorem 3.4.4 of \cite{pavlov-scholbach-spectra}. The rest of the proof above works without changes, as \cite{gutierrez-transfer-quillen} works for any enriching category, even $\M$ itself.

Lastly, note that if $R_K$ is an $\M$-enriched colocalization (i.e. defined via internal hom objects rather than simplicial mapping spaces) then $\rkm$ is a monoidal model category by Remark 5.3 in \cite{gutierrez-transfer-quillen}, and $K$-colocal objects are closed under the smash product by Lemma 5.6 of \cite{gutierrez-transfer-quillen}.
\end{proof}

\begin{example}
Let $K = \{ \Sigma^{n+1}S^0\}$ where $S^0$ is the sphere spectrum. Then $K$-colocal spectra are precisely the $n$-connective spectra, and $K$-colocalization amounts to taking $n$-connective covers. This colocalization may be defined as an $\sset$-enriched colocalization, but not as an $\M$-enriched colocalization since $K$-colocal objects are not closed under smash product (see Section 7 of \cite{gutierrez-transfer-quillen}). However, for $n = 0$, $K$-colocal objects \textit{are} closed under smash product. Thus, Theorem \ref{thm:spectra-preservation} demonstrates that taking connective covers preserves $P$-algebra structure for every colored operad $P$, that the connective cover colocalization lifts to homotopy categories of $P$-algebras, and that the forgetful functor preserves connective cover colocalizations.
\end{example}

\subsection{Spaces}

Consider the Quillen model structure on $\Top$, the category of compactly generated topological spaces. 

\begin{example}[$n$-connected covers] \label{ex:n-conn-covers}
Let $K = \{S^m\;|\; m>n\}$, so that $R_K(X) = CW_A(X)$ where $A=S^n$ \cite{chacholski-thesis,farjoun}. The $K$-colocal objects are $X$ with $\pi_{\leq n}(X) = 0$, and the $K$-colocal equivalences are maps $f$ with $\pi_{>n}(f)$ an isomorphism. As this set $K$ is closed under smashing with spheres, Theorems 4.3, 4.5, and 9.3 in \cite{white-yau2} demonstrate that the pushout product axiom is satisfied in $R_K(\Top)$. The same is true for any CW complex $A$. For the case $K = \{S^1\}$, the $K$-colocal spaces are precisely those $X$ with $\pi_0(X) = 0$; i.e., $X$ is path connected. Any $E_\infty$-operad $\sO$ has path connected spaces (contractible even), so is entrywise $K$-colocal. It is easy to check that such an operad is in fact $\Sigma$-cofibrant in the $K$-colocal model structure on symmetric sequences, since the fixed-point property of $E\Sigma_n$ guarantees the existence of an equivariant lift in a lifting problem against a $K$-colocally trivial fibration. Theorem 6.2 in \cite{white-yau2} proves that $R_K$ preserves $E_\infty$-algebras (the requisite smallness assumptions are verified in \cite{white-topological}). Theorem \ref{equivalent.preservation} now demonstrates that $R_K$ lifts to the homotopy category of $E_\infty$-algebras in such a way that the forgetful functor $U$ preserves right Bousfield localization.
\end{example}

\subsection{Equivariant spaces}

Let $G$ be a compact Lie group, and let $\M = \Top^G$ denote the category of $G$-equivariant compactly generated spaces, with the fixed-point model structure where a map $f$ is a weak equivalence (resp. fibration) if and only if $f^H$ is a weak equivalence (resp. fibration) for every closed $H < G$. Consider the model structure on operads valued in $\M$ transferred along the free-operad functor from the product model structure on $\M^\Sigma = \prod_{n \geq 0} \M^{\Sigma_n}$, where each $\M^{\Sigma_n}$ has the projective model structure. Note that this is usually the wrong model structure for the study of $G$-equivariant operads (because one wants to study fixed spaces of subgroups of $G\times \Sigma_n$), but this model structure is the natural home for the $G$-equivariant $E_\infty$-operad of \cite{lms}. In this model structure, the cofibrant replacement of the commutative operad $Com$ is an equivariant $E_\infty$-operad, which plays an important role in the search for algebraic models for equivariant spectra. 

We will show a preservation result for this operad. Note, however, that this is the wrong operad to encode complete equivariant commutativity (including norms) in $G$-spectra, because it does not allow for mixing of the $G$-action with the $\Sigma_n$-action. The correct operads to study for norms are the $N_\infty$-operads of \cite{blumberg-hill}.

\begin{example}
Suppose $F$ is a nonempty set of subgroups of $G$ and let $K(F) = \{(G/H)_+ \;|\; H\in F\}$. Colocalization with respect to $K(F)$ preserves algebras over any $G-E_\infty$-operad $\mathscr{E}$, because $\mathscr{E}$ is $\Sigma$-cofibrant in $\prod_n (\Top^{G})^{\Sigma_n}$, so $\mathscr{E}$-algebras inherit a semi-model structure in $\Top^G$ and in $R_{K(F)}(\Top^G)$. Theorem \ref{equivalent.preservation} now demonstrates that the colocalization $R_{K(F)}$ (that focuses on cells $G/H$ for $H\in K(F)$) lifts to the homotopy category of $E_\infty$-algebras and that the forgetful functor preserves $R_{K(F)}$.
\end{example}

\begin{example}
Let $K = \{S^{n+1}\}$, so that $K$-colocal objects are $n$-connected covers. Then both the operad $\mathscr{E}$ and the $N_\infty$ operads of \cite{blumberg-hill} are objectwise $K$-colocal and $\Sigma_n$-free. They therefore satisfy condition $\star^\sO$ from \cite{white-yau2}, and hence their algebras are preserved by taking $n$-connected covers. Theorem \ref{equivalent.preservation} demonstrates that the $n$-connected cover colocalization lifts to the homotopy categories of $E_\infty$ and $N_\infty$-algebras.
\end{example}

\subsection{Chain complexes}

\begin{theorem} \label{thm:app-chains}
Let $k$ be a field of characteristic zero. Then, for any set of cofibrant objects $K$, and any colored operad $\sO$ that is objectwise $K$-colocal, $R_K$ lifts to the homotopy category of $\sO$-algebras and the forgetful functor $U:\alg_{\sO}(\Chk)\to \Chk$ preserves $R_K$.
\end{theorem}

\begin{proof}
Theorem 11.7 in \cite{white-yau2} demonstrates that every right Bousfield localization $R_K(\Chk)$ is monoidal. Theorem 11.12 in \cite{white-yau2} demonstrates that $\Chk$ satisfies condition $\diamond$, which says all symmetric sequences are projectively cofibrant, since $k$ has characteristic zero. Thus, any $\sO$ that is objectwise $K$-colocal is $\Sigma$-cofibrant with respect to $R_K(\Chk)$, so there is a transferred semi-model structure on $\sO$-algebras. The preservation theorem (Theorem 6.2 in \cite{white-yau2}) demonstrates that $R_K$ preserves $\sO$-algebras, and Theorem \ref{equivalent.preservation} provides the required lift of $R_K$ to $\sO$-algebras.
\end{proof}

Let $S(n)$ denote the chain complex that is $k$ in degree $n$ and 0 elsewhere.

\begin{example} \label{ex:n-conn-cover-chains}
Let $K = \{S(n)\}$ for some $n$. Then the $K$-colocal objects are the $X$ such that $H_{<n}(X) = 0$, and the $K$-colocal equivalences are maps $f$ such that $H_{\geq n}(f)$ is an isomorphism. The functor $R_K$ can be viewed as an $n$-connected cover. Suppose $k$ has characteristic zero. For any $n$, $R_K$ lifts to the homotopy category of $\sO$-algebras, for any $E_\infty$-operad $\sO$, by Theorem \ref{thm:app-chains}, because all spaces $\sO(n)$ have $H_i(\sO(n)) = 0$ for all $i$, hence are $K$-colocal. Similarly, $R_K$ lifts to the homotopy category of commutative differential graded algebras when $k$ has characteristic zero.
\end{example}

\begin{example}
Suppose $R_K(\Chk)$ is a monoidal model category and that the $K$-colocalization functor can be chosen to be lax monoidal (e.g. see \cite{gutierrez-transfer-quillen} (5.6)). Then for any colored operad $\sO$, the sequence $R_K\sO$ defined by
\[
(R_K\sO)\smallbinom{d}{[\uc]} = R_K\left(\sO \smallbinom{d}{[\uc]}\right)\]
is a colored operad over $\rkm$. By construction, this colored operad is objectwise $K$-colocal, so the right Bousfield localization $R_K$ lifts to the homotopy category of $R_K\sO$-algebras.
\end{example}

\begin{remark}
The category of simplicial abelian groups has a cofibrantly generated model structure \cite{quillen} in which all objects are fibrant. This category is equivalent to the category of bounded below chain complexes, by the Dold-Kan Theorem. The normalized chains functor $N$ is a natural isomorphism, compatible with the model structures, and is monoidal by \cite{schwede-shipley-equivalences} (4.1). Thus, all results about right Bousfield localizations lifting to categories of algebras in $\Chk$ yield analogous results in the category of simplicial $k$-modules.
\end{remark}

\subsection{Stable module category}

\begin{example}
Let $\M = k[G]$-mod and let $K = \{k\}$ where $k$ has the trivial $G$ action. The colocalization of $\M$ with respect to this $K$ is studied in \cite{beligiannis-reiten} (IV.2.7). Since $k$ is the monoidal unit, the commutative operad is objectwise $K$-colocal. Let $K'$ denote the stable, monoidal closure of $K$, following \cite{barnes-roitzheim-stable}, and let $K''$ be the closure of $K'$ and $\{k[\Sigma_n]\;|\;n=1,2,...\}$. Corollary 12.3 of \cite{white-yau2} demonstrates model structures on commutative and associative algebras in this setting, so Theorem 6.2 in \cite{white-yau2} provides preservation results for such algebraic structures under right Bousfield localization. Theorem \ref{equivalent.preservation} then proves that $R_{K'}$ lifts to the homotopy category of commutative monoids and that $R_{K''}$ lifts to the homotopy categories of commutative monoids and associative monoids. 
\end{example}


\end{document}